\documentclass[12pt,reqno]{amsart}

\newtheorem{lemma}{Lemma}[section]
\newtheorem{definition}{Definition}[section]
\newtheorem{proposition}{Proposition}[section]
\newtheorem{theorem}{Theorem}[section]
\newtheorem{remark}{Remark}[section]

\newtheorem{corollary}{Corollary}[section]

\newcommand{\R}{\mathds{R}}
\newcommand{\C}{\mathds{C}}

\newcommand{\dt}{\mathsf{dt}}
\newcommand{\dx}{\frac{\partial}{\partial x}}
\newcommand{\dy}{\frac{\partial}{\partial y}}
\newcommand{\dxu}{\frac{\partial}{\partial x_1}}
\newcommand{\dxd}{\frac{\partial}{\partial x_2}}
\newcommand{\dxt}{\frac{\partial}{\partial x_3}}
\newcommand{\dxc}{\frac{\partial}{\partial x_4}}
\newcommand{\di}{\mathsf{d}}

\usepackage{amssymb,amsmath,amsthm}
\usepackage{epsfig,amssymb,amsmath,amsthm,graphicx,psfrag}
\usepackage[all]{xy}
\usepackage{enumerate}
\usepackage{color}
\usepackage{dsfont}
\usepackage{amsfonts, latexsym,amsxtra, amstext}
\usepackage{mathrsfs}
\usepackage{multicol}
\usepackage{lscape}

\usepackage{tikz-cd}
\begin{document}

\title{Flat Affine and Symplectic Geometries on Lie groups }

 \author{Villab\'on, Andr\'es.*}

\date{\today}

\begin{abstract} 
In this paper we exhibit a family of flat left invariant affine structures on the double Lie group of the oscillator Lie group of dimension 4, associated to each solution of classical Yang-Baxter equation given by Boucetta and Medina. On the other hand, using Koszul's method, we prove the existence of an immersion of Lie groups between the group of affine transformations of a flat affine and simply connected manifold and the classical group of affine transformations of $\R^n$. In the last section, for each flat left invariant affine symplectic connection on the group of affine transformations of the real line, describe for Medina-Saldarriaga-Giraldo, we determine the affine symplectomorphisms. Finally we exhibit the Hess connection, associated to a Lagrangian bi-foliation, which is flat left invariant affine.     
\end{abstract}
\maketitle


\vskip5pt
\noindent
 * Instituto de Matem\'aticas, Universidad de Antioquia, Medell\'in-Colombia

 e-mails: edgar.villabon@udea.edu.co

\maketitle

\section{Introduction}
Let $M$ be a (paracompact) connected real smooth manifold of dimension $n$. A \textbf{\textit{linear frame $u$}} at point $x\in M$ is an ordered basis $(X_1,\ldots,X_n)$ of the tangent space $T_x(M)$ it will usually be identified with a linear isomorphism from $\R^n$ onto $T_x(M).$ The set $L(M)$ of all linear frames at all points of $M$ is a smooth principal fiber bundle over $M$ with group $\mathsf{GL}(n,\R),$ the right action of $\mathsf{GL}(n,\R)$ on $L(M)$ is given by
\begin{equation*}
u\cdot a=u\circ a\quad\mbox{with $u\in L(M)$ and $a\in\mathsf{GL}(n,\R)$.}
\end{equation*}
The diffeomorphism of $L(M)$ associated to $a$ will be denoted $R_a (u)=u\cdot a.$\\
\\ 
The \textbf{\textit{canonical form}} $\theta$ of $L(M)$ is the $\R^n$-valued 1-form defined by 
\begin{equation*}
\theta(X_u):=u^{-1}(\pi_{\ast,u}(X_u))\hspace{0.4 cm}\mbox{for $X_u\in T_u(L(M))$}
\end{equation*}
where $\pi:L(M)\rightarrow M$ in the natural projection.\\
\\
In what follows we will privilege the notion of connection due to C. Ehresmann \cite{Eh}. When necessary we will also use the notions of E. Cartan \cite{C1} or  J.L. Koszul \cite{K1}.\\ 
\\
A \textit{\textbf{Linear connection}} $\Gamma$ over $M$  is a smooth distribution $H$ in $L(M)$ transverse and supplementary to the fiber and invariant under the right action of $\mathsf{GL}(n,\R),$ defined above. The subspace $H_u,$ where $u$ is a linear frame at $ x\in M,$ is called the horizontal space in $u.$
\\
It is clear that every vector $X_u\in T_u(L(M))$ can be uniquely written as $X_u=X_{u,H}+X_{u,V},$ with $X_{u,H}\in H_u$ and $X_{u,V}$ tangent to the fiber over $u.$\\
\\
Each $\vec{x}\in\R^n$ determines a \textbf{\textit{standard horizontal vector field}} $B(\vec{x})_u$ defined as the unique element of $H_u$ such that $\pi_{\ast,u}(B(\vec{x})_u)=u(\vec{x}).$\\
Moreover, $\Gamma$ defines a unique 1-form $\omega$ on $L(M)$ with values in  $\mathfrak{g}=\mathsf{gl}(n,\R),$ called the \textbf{\textit{connection form}}. More precisely $\omega_u(X_u)$ is the unique $A\in\mathfrak{g}$ such that $A^*_u=X_{u,V}$ where $$A^*_u=\left.\frac{\di}{\dt}\right|_{t=0} u\cdot\exp tA$$
is the \textit{fundamental vector field corresponding to} $A.$\\
The connection form $\omega$ satisfies the following conditions 
\begin{align}
\omega(A^*)&=A\quad \mbox{for all $A\in\mathfrak{g},$}\label{condicion1}\\
(R_{\sigma})^*\omega&=\mathsf{ad}(\sigma^{-1})\omega \quad\mbox{for all $\sigma\in\mathsf{GL}(n,\R),$}\label{condicion2}
\end{align}  
where $\mathsf{ad}$ denotes the adjoint representation of $\mathsf{GL}(n,\R).$ Having a connection $\Gamma$ is equivalent to have a 1-form $\omega$ on  $L(M)$ with values in $\mathfrak{g}$ verifying \eqref{condicion1} y \eqref{condicion2}. A smooth vector field $X\in\mathfrak{X}(L(M))$ such that $\omega(X)=0,$ is called \textbf{\textit{horizontal vector}} relative to $\Gamma$.
\\
\\
The \textbf{\textit{curvature form}} $\Omega$ and the \textbf{torsion form} $\Theta$ of $(L(M),\Gamma)$ are defined respectively by
$$\Omega_u(X_u,Y_u)=(\di\omega)_u(X_{u,H},Y_{u,H})$$
$$\Theta_u(X_u,Y_u)=(\di\theta)_u(X_{u,H},Y_{u,H})$$  
with $X_u,Y_u\in T_u(L(M))$ and $u\in L(M).$
\\
Note that $\Omega$ takes values in $\mathsf{gl}(n,\R)$ and $\Theta$  in $\R^n.$ The following equations are called \textbf{\textit{structure equations}} (also known as Cartan equations) (see \cite {KN}),

\begin{align}
(\di\theta)_u(X_u,Y_u)&=-\frac{1}{2}(\omega_u(X_u)\theta_u(Y_u)-\omega_u(Y_u)\theta_u(X_u))+\Theta_u(X_u,Y_u),\label{estructural1}\\
(\di\omega)_u(X_u,Y_u)&=-\frac{1}{2}[\omega_u(X_u),\omega(Y_u)]+\Omega_u(X_u,Y_u)\label{estructural2}
\end{align}
\\
\\
Given a global (or local) diffeomorphism  $f$ of $M$ we will denote $\widetilde{f}$ to the global (or local) diffeomorphism  of $L(M)$ defined by $\widetilde{f}(u)=(f_{\ast,x}(X_1),\ldots,f_{\ast,x}(X_n))$ with $u=(X_1,\ldots,X_n)$ a linear frame at $x\in M.$ The diffeomorphism $\widetilde{f}$ is called \textbf{\textit{ the natural lift}} of $f$ to $L(M).$
\\
When $f$ satisfies the condition $\widetilde{f}^{\ast}\omega=\omega,$ it is called a global (or local) \textbf{\textit{affine transformation}}  of $(M,\Gamma)$.\\  

If $M$  is simply connected endowed with a linear connection $\Gamma$, it is well known that in $\Gamma$ is flat, that is, without curvature, the connection over $M$ is isomorphic to the connection $\Gamma'$ given by the distribution $H$  in $(M\times\mathsf{GL}(n,\R),\psi)$ where $H_u$ is the tangent space to $M\times\{a\}$ in $u=(x,a)\in M\times\mathsf{GL}(n,\R).$ The connection $\Gamma'$ is called the \textbf{\textit{canonical flat (linear) connection}} over $M$ (see \cite{KN}).\\
Let $\pi_2:M\times\mathsf{GL}(n,\R)\rightarrow \mathsf{GL}(n,\R)$ natural projection. Notice that the connection form of the canonical flat linear connection is given by 
$$\omega'=\pi_2^*\theta^+$$
where $\theta^+$ is  the left invariant 1-form  over  $\mathsf{GL}(n,\R),$ with values in  $\mathsf{gl}(n,\R)$ defined by $$\theta^+(A^+):=A^+$$
and $A^+\in\mathfrak{g}$ is the left invariant vector field corresponding to $A\in T_{\varepsilon}\mathsf{GL}(n,\R).$ The 1-form $\theta^+$ is called the canonical 1-form over $\mathsf{GL}(n,\R)$.
\\
\\
If $M=G$ is a Lie group and the left multiplications $L_\sigma$ (respectively right  multiplications $R_\sigma$) are affine transformations it is said that $\Gamma$ is \textbf{\textit{left invariant (respectively right invariant)}}. This means that the lift $\widetilde{L}_\sigma$ of  ${L}_\sigma$  (respectively $\widetilde{R}_\sigma$) preserves the  distribution $H$ given by $\Gamma.$ In this case we denote the connection by $\Gamma^+$ and we will say that $(G,\Gamma^+)$ is  an affine Lie group.
\\
\\
It is known that having a linear connection $\Gamma$ over $M$ is equivalent to having a covariant derivative  $\nabla$ (see \cite{KN}). Let us consider the product \begin{equation}\label{Eq:productinducedbyaconnection}XY:=\nabla_XY,\quad\text{for}\quad X,Y\in \mathfrak{X}(M).\end{equation}  If the curvature  tensor relative to $\nabla$ is identically zero, this product verifies the condition
\[ [X,Y]Z=X(YZ)-Y(XZ), \quad\text{for all}\quad X,Y,Z\in \mathfrak{X}(M).\]
If both torsion and curvature vanish identically we have
\begin{equation} \label{Eq:leftsymmetricproduct} (XY)Z-X(YZ)=(YX)Z-Y(XZ), \quad\mbox{for all}\quad X,Y,Z\in \mathfrak{X}(M).\end{equation} 
A vector space endowed with a bilinear product satisfying Equation \eqref{Eq:leftsymmetricproduct} is called a \textbf{\textit{left symmetric algebra.}} Left symmetric algebras are Lie algebras with the bracket given by $[X,Y]_1:=XY-YX$. If $\nabla$ is torsion free this Lie bracket agrees with the usual Lie bracket of $\mathfrak{X}(M)$, that is \begin{equation}\label{equliebrakettorsionfree}
[X,Y]=XY-YX
\end{equation}
Recall that product \eqref{Eq:productinducedbyaconnection} satisfies $(fX)Y=f(XY)$ and $X(gY)=X(g)Y+g(XY)$, for all $f,g\in C^\infty(M,\R)$, i.e., the product defined above is $\R$-bilinear and $C^\infty(M,\R)$-linear in the first component. \\ 
If $M=G$ is a Lie group and $\nabla=\nabla^+$ is left invariant, it is well known that $(G,\nabla^+)$ is flat affine if and only if Product \eqref{Eq:productinducedbyaconnection} turn $\mathfrak{g}$ into a left symmetric algebra (see \cite{M2}). 
If this is the case, we denote by $\omega^+$ the corresponding connection form associated to the connection $\nabla^+.$\\

\section{Flat affine geometry on 4-dimensional oscillator Lie group}
The $3$-dimensional Heisenberg Lie group is the real manifold  $\mathsf{H}_3=\R\times\C$ with product 
$$(s,w)(s',w')=\left(s+s'+\frac{1}{2}\mathsf{Im}(\overline{w}w'),w+w'\right)$$ 
\begin{definition} Let $\lambda$ be a positive real number and $\rho:\R\rightarrow\mathsf{Aut}(\mathsf{H}_3)$ the Lie group homomorphism given by 
	$$\rho(t)(s,z)=(s,e^{it\lambda}z).$$ 
	The real $\lambda$-oscillator group of dimension $4$, denoted by $O_{\lambda}$, is the semidirect product of Heisenberg Lie group $\mathsf{H}_3$ for the group $(\R,+)$ by the representation $\rho.$ 
\end{definition}
The following result is known (see \cite{MR2}).
\begin{lemma}
	The Lie group $ O_{\lambda}$ is independent of $\lambda,$ for  $\lambda>0.$
\end{lemma}
\begin{proof}
The statement follows from the fact that the map $\phi:\mathsf{H}_3\rightarrow\mathsf{H}_3$ defined by $\phi(s,z)=\left(s,z^{\frac{\lambda'}{\lambda}}\right),$  is an automorphism of Lie groups such that for all $t\in\R$ the following diagram commutes
\begin{equation*}
\begin{tikzcd}
\mathsf{H}_3 \arrow{r}{\rho(t)} \arrow{d}{\phi}
&\mathsf{H}_3 \arrow{d}{\phi}\\
\mathsf{H}_3 \arrow{r}{\rho'(t)} &\mathsf{H}_3
\end{tikzcd}
\end{equation*}
where $\rho$ and $\rho'$ are the representations
\begin{align*}
\begin{array}{cccc}
\rho:&\R&\rightarrow&\mathsf{Aut}(\mathsf{H}_3)\\
&t&\mapsto&(s,e^{\lambda ti})
\end{array}&  &and&& \begin{array}{cccc}
\rho':&\R&\rightarrow&\mathsf{Aut}(\mathsf{H}_3)\\
&t&\mapsto&(s,e^{\lambda' ti})
\end{array}
\end{align*}
\end{proof}
In the follows we denote by $O$ the oscillator Lie group of dimension $4$ with $\lambda=1.$ Let $(e_0,\ e_1, \ e_2,\ d)$ be the natural basis for $\R^4.$ The corresponding left invariant vector fields are as follows:
\begin{align}\label{Eq:basisofleftinvariant}
e_{0}^+&=\frac{\partial}{\partial x_1}\\ \notag
e_{1}^+&=\left(\frac{1}{2}(x_2\sin(x_4)-x_3\cos( x_4))\right)\frac{\partial}{\partial x_1}+\cos(x_4)\frac{\partial}{\partial x_2}+\sin(x_4)\frac{\partial}{\partial x_3}\\ \notag
e_{2}^+&=\left(\frac{1}{2}(x_2\cos(x_4)+x_3\sin(x_4))\right)\frac{\partial}{\partial x_1}-\sin(x_4)\frac{\partial}{\partial x_2}+\cos(x_4)\frac{\partial}{\partial x_3}\\ \notag
d^+&=\frac{\partial}{\partial x_4}.
\end{align}
It is directly verified that the non zero Lie brackets, except antisymmetry, are given by
\begin{align}\label{Eq:corcheteoscilador}
[e_1^+,e_2^+]=e_0^+, && [d^+,e_1^+]=e_2^+, && [d^+,e_2^+]=-e_1^+.
\end{align}
consequently the Lie algebra of $O$ is the real vector space with basis $\{e_0^+,e_1^+,e_2^+,d^+\}$ and  bracket  defined in  \eqref{Eq:corcheteoscilador}. This Lie algebra will be denoted by $\mathfrak{o}$.

\subsection{Flat affine structures on $O$} To construct these kind of structures, we start by finding flat left invariant affine structures on the group of isometries of the Euclidean plane $E(2)$, called the Euclidean group. Then we use the cohomology of left symmetric algebras developed in \cite{Ni} to extend these structures to the oscillator group of dimension $4$.\\
\\
Notice that the oscillator group can be viewed as an extension of the Euclidean group $E(2)$. This extension is given by the exact split sequence of Lie groups
$$\xymatrix{
	\{0\}\ar[r]&(\R,+)\ar[r]^(0.6){i}&O\ar[r]^(0.4){\pi}&E(2)\ar[r]&\{I\}
}$$
where $i(s)=(s,0,0),$ for all $s\in\R$ and 
\begin{align*}
\pi(s,x+iy,t)=\left(
\begin{array}{ccc}
\cos(t)&-\sin(t)&x\\
\sin(t)&\cos(t)&y\\
0&0&1
\end{array}\right),&&\mbox{for all $(s,x+iy,t)\in O,$}
\end{align*}
and the section is given by the Lie group homomorphism   $\mathbf{s}:E(2)\rightarrow O$  defined by $\mathbf{s}(A,(x,y))=(0,x+iy,t)$ where $A=\left(
\begin{matrix}
\cos(t)&-\sin(t)\\
\sin(t)&\cos(t)
\end{matrix}\right).$ \\ The Lie algebra of the Euclidean Lie group, $\mathfrak{e}(2)=\mathsf{Vect}_{\R}\{e_1,e_2,d\},$ has  Lie bracket given by $[d,e_1]=e_2$, $[d,e_2]=-e_1$, and $[e_1,e_2]=0.$
\begin{proposition}\label{PSGE2}
	The following left symmetric products determine flat left invariant non-isomorphic affine structures on the Euclidean group $E(2).$ 
	\begin{align}
	\mathcal{F}_1(\alpha)&& &\begin{array}{c|ccc}
	&e_1&e_2&d\\
	\hline
	e_1&0&0&\alpha e_1\\
	e_2&0&0&\alpha e_2\\
	d&\alpha e_1+e_2&-e_1+\alpha e_2&\alpha d
	\end{array},\hspace{0.2cm} \label{Producto1}
	\end{align}
	\begin{align}
	\mathcal{F}_2(\alpha)&& &\begin{array}{c|ccc}
	&e_1&e_2&d\\
	\hline
	e_1&0&0&0\\
	e_2&0&0&0\\
	d&e_2&-e_1&\alpha d
	\end{array},\quad\mbox{with $\alpha\in\R\setminus\{0\}$}\label{Producto2}
	\end{align}
	\begin{align}
	\mathcal{F}_3&& &\begin{array}{c|ccc}
	&e_1&e_2&d\\
	\hline
	e_1&d&0&0\\
	e_2&0&d&0\\
	d&e_2&-e_1&0
	\end{array}\label{Producto3}\\
	\mathcal{F}_4&& &\begin{array}{c|ccc}
	&e_1&e_2&d\\
	\hline
	e_1& e_1+2e_2-d&0& e_1+2e_2- d\\
	e_2&0& e_1+2e_2-d&2e_1+4e_2-2d\\
	d&e_1+3e_2- d&e_1+4e_2-2d&3e_1+11e_2-5d
	\end{array}\label{Producto4}&
	\end{align}
\end{proposition}
\begin{proof}
Recall that finding flat left invariant affine structures on $E(2)$ is equivalent to finding bilinear products on $\mathfrak{e}(2)$ satisfying \eqref{Eq:leftsymmetricproduct} and \eqref{equliebrakettorsionfree}.\\ 
Let us denote by $L_a$  the endomorphism of the vector space $\mathfrak{e}(2)$ defined by $L_a(b)=ab$.\\
If a bilinear product is compatible with the Lie bracket of $\mathfrak{e}(2)$, the endomorphisms $L_{e_1}$, $L_{e_2}$ and $L_d$ can be writen, in the base  $(e_1,e_2,d)$, as
\begin{align}\label{productoE}
L_{e_1}=(a_{ij}),&&
L_{e_2}=\left(\begin{array}{ccc}
a_{12}&b_{12}&b_{13}\\
a_{22}&b_{22}&b_{23}\\
a_{32}&b_{32}&b_{33}
\end{array}\right),&&
L_d=\left(\begin{array}{ccc}
a_{13}&b_{13}-1&c_{13}\\
1+a_{23}&b_{23}&c_{23}\\
a_{33}&b_{33}&c_{33}
\end{array}\right)
\end{align}
where $a_{ij},\ b_{ij}$ y $c_{ij}$ are real numbers.\\
Moreover, if the product on $\mathfrak{e}(2)$ 
satisfies \eqref{Eq:leftsymmetricproduct}, the Endomorphisms  \eqref{productoE} satisfy the system of equations 
\begin{align*}
[L_{e_1},L_{e_2}]&=0\\
[L_d,L_{e_1}]&=L_{e_2}\\
[L_d,L_{e_2}]&=-L_{e_1}
\end{align*}
Solving the system and using the above observations we obtain the products described in the tables.
\end{proof}

\subsection{Left symmetric algebra cohomology}
Let $A$ be a left symmetric algebra and $M$ a $A$-bimodulo. We denote by $H_{SG}^k(A,M)$ the cohomology group of order $k$ of $A$ with coefficients in $M.$ Let us consider $\R$ as a trivial $A$-bimodulo. A scalar 2-cocycle of $A$ is a bilinear map $f:A\times A\rightarrow\R$  verifying
$$f(ab-ba,c)=f(a,bc)-f(b,ac).$$ 
A scalar 2-coboundary of $A$ is a 2-cocycle $f$ such that  $f(a,b)=\varphi(ab),$ for some   $\varphi\in A^*$. The cohomology group  $H_{SG}^2(A,\R)=Z^2/B^2,$ where  $Z^2$ (respectively $B^2$) is the set of 2-cocycles (respectively 2-coboundaries), 
classifies the left symmetric algebras extensions of $ A $ for which $\R$ is an ideal of null square (see \cite{Ni}). In particular
$$0\rightarrow\R\rightarrow\R\times A\rightarrow A\rightarrow0$$
is an exact sequence of left symmetric algebras.\\
\begin{proposition}
	The 2-order cohomology groups of left symmetric algebras of $\mathfrak{e}(2)$ corresponding to the products $\mathcal{F}_1(0),\ \mathcal{F}_1(\alpha),$ $\mathcal{F}_2(\alpha),$ $\mathcal{F}_3$ y $\mathcal{F}_4$  are respectively  given by
	\begin{align*}
	H_{SG,1,0}^2(\mathfrak{e}(2),\R e_0)&\cong\mathsf{Vect}_{\R}\{E_{12}-E_{21},E_{22}+E_{11},E_{33}\}.\\
	H_{SG,1,\alpha}^2(\mathfrak{e}(2),\R e_0)&\cong\{0\}, \ \mbox{si $\alpha\neq0$}.\\
	H_{SG,2,\alpha}^2(\mathfrak{e}(2),\R e_0)&\cong\mathsf{Vect}_{\R}\{E_{11}+E_{22},E_{12}-E_{21}\}\\
	H_{SG,3}^2(\mathfrak{e}(2),\R e_0)&\cong\mathsf{Vect}\{E_{12}-E_{21}\}\\
	H_{SG,4}^2(\mathfrak{e}(2),\R e_0)&\cong\mathsf{Vect}\{E_{12}-E_{21}+2E_{13}-E_{23}\}.
	\end{align*}
	where $\{E_{ij}\}$ denote the elementary matrices on  $\mathsf{M}_{3\times3}(\R).$ 
\end{proposition}
\begin{proof}
Let us consider $M=\R e_0$ as a trivial $\mathfrak{e}(2)$-bimodule, with $\mathfrak{e}(2)=\mathsf{Vect}_{\R}(e_1,e_2,d).$ Each of the following statements can be verified by direct calculations.\\
In the algebra with product given by $\mathcal{F}_1(0),$ the scalar  $2$-cocycles of $\mathfrak{e}(2)$ with coefficients in $M,$ are of the form $$g=p(E_{11}+E_{22})+q(E_{12}-E_{21})+rE_{31}+sE_{32}+tE_{33}$$
and the 2-coboundaries as $$f=rE_{31}+sE_{32}.$$
In the algebras corresponding to the products $\mathcal{F}_1(\alpha)$, with $\alpha\neq0$, the $2$-cocycles (respectively the 2-coboundaries), are given by
\begin{gather*}
g=p\left(E_{13}+E_{31}-\frac{1}{\alpha}E_{32}\right)+q\left(E_{23}+\frac{1}{\alpha}E_{31}+E_{32}\right)+rE_{33}
\end{gather*}
and $$f=p(\alpha E_{13}+\alpha E_{31}-E_{32})+q(\alpha E_{23}+E_{31}+\alpha E_{32})+rE_{33}.$$
\\
In the algebras with product $\mathcal{F}_2(\alpha)$ the 2-cocycles  (respectively the 2-coboundaries) are written $$g=t(E_{11}+E_{22})+s(E_{12}-E_{21})+rE_{31}+pE_{32}+qE_{33}$$ and  $$f=rE_{31}+pE_{32}+qE_{33}.$$
\\
The 2-cocycles (respectively the 2-coboundaries) associated to the algebra with product $\mathcal{F}_3,$ are given by $$g=t(E_{11}+E_{22})+s(E_{12}-E_{21})+rE_{31}+pE_{32}$$ and  $$f=t(E_{11}+E_{22})+rE_{31}+pE_{32}.$$
\\
Finally, the scalar 2-cocycles  (respectively 2-coboundaries) of the left symmetric algebra defined  by $\mathcal{F}_4,$ are of the form 
\begin{gather*}
g=p(E_{11}+E_{22}+E_{13}+2E_{23})+q(E_{12}-E_{21}+2E_{13}-E_{23})+\\r(E_{31}+ E_{33})+s(E_{32}+2E_{33})
\end{gather*} and $$f=p(E_{11}+E_{22}+E_{13}+2E_{23})+r(E_{31}+E_{33})+s(E_{32}+2E_{33}),$$
where $p,q,r,s,t$ are real numbers.\\
The previous statements  imply the results  enunciated in the proposition.
\end{proof}
\textit{Each class $[g]$ of the left symmetric 2-cohomology of $\mathfrak{e}(2)$ determines a left symmetric product in $\mathfrak{o}\cong\R e_0\rtimes_g\mathfrak{e}(2)$  by the formula $(e_0,a)\cdot(e_0,b)=(g(a,b),a\cdot b),$ where $g$ is a class representative $[g]$.} This method and other observations lead to the following result.   

\begin{theorem}
	The left symmetric products on the Lie algebra $\mathfrak{o}$ given by the following tables determine  non isomorphic flat left invariant affine connections on the 4-dimensional oscillator group 
	\begin{align}
	\mathcal{F}_1(0,t,s,q)&& &\begin{array}{c|cccc}
	&e_0&e_1&e_2&d\\
	\hline
	e_0&0&0&0&0\\
	e_1&0&t e_0&qe_0&0\\
	e_2&0&-qe_0&t e_0&0\\
	d&0&e_2&-e_1&s e_0
	\end{array}
	\\
	\mathcal{F}_2(\alpha,t,q) && &\begin{array}{c|cccc}
	&e_0 &e_1&e_2&d\\
	\hline
	e_0&0&0&0&0\\
	e_1&0&te_0&qe_0&0\\
	e_2&0&-qe_0&te_0&0\\
	d&0&e_2&-e_1&\alpha d
	\end{array}
	\\
	\mathcal{F}_3(q)&& &\begin{array}{c|cccc}
	&e_0   &e_1&e_2&d\\
	\hline
	e_0&0&0&0&0\\
	e_1&0&d&qe_0&0\\
	e_2&0&-qe_0&d&0\\
	d&0&e_2&-e_1&0
	\end{array}
	\end{align}
	\begin{align}
	&\begin{array}{c|cccc}
	\mathcal{F}_4(q) &e_0&e_1&e_2&d\\
	\hline
	e_0&0&0&0&0\\
	e_1&0&\frac{1}{2}e_1+e_2-d&qe_0&qe_0+\frac{1}{4}e_1+\frac{1}{2}e_2-\frac{1}{2}d\\
	e_2&0&-qe_0&\frac{1}{2}e_1+e_2-d&-\frac{q}{2}e_0+\frac{1}{2}e_1+e_2-d\\
	d&0&qe_0+\frac{1}{4}e_1+\frac{3}{2}e_2-\frac{1}{2}d&-\frac{q}{2}e_0-\frac{1}{2}e_1+e_2-d&-\frac{3}{8}e_1+\frac{7}{4}e_2-\frac{5}{4}d
	\end{array}
	\end{align}
	where $t,s\in[0,+\infty),$ $\alpha\in\R^*$ and $q=\frac{1}{2}.$
\end{theorem}

\begin{corollary}
	The flat left invariant affine connections on the 4-dimensional oscillator group  given by the products  $\mathcal{F}_1\left(0,t,s,\frac{1}{2}\right), \mathcal{F}_3\left(\frac{1}{2}\right)$ and $\mathcal{F}_4\left(\frac{1}{2}\right)$ are geodesically  complete.
\end{corollary}
\begin{proof}
This results follows from the fact that each left symmetric algebra has right multiplications $R_a$ of trace zero (see \cite{H} and \cite{M2}).
\end{proof}
\begin{remark}
	It is important to notice that left symmetric products\\ $\mathcal{F}_1(0,t,s,q),$ $\mathcal{F}_2(\alpha,t,q),$ $\mathcal{F}_3(q)$ and $\mathcal{F}_4(q),$ with $q\neq0$ are isomorphic respectively to\\ $\mathcal{F}_1(0,s,t,1/2),$  
	$\mathcal{F}_2(\alpha,t,1/2),$ $\mathcal{F}_3(1/2)$ y $\mathcal{F}_4(1/2).$ 
\end{remark}

\subsection{Double Lie groups of the oscillator group $O$ relative to a solution of classical Yang-Baxter equation}\vspace{12pt}

The aim of this section is to determine the dual Lie groups of the oscillator group $O$ relative to solutions of the classical Yang-Baxter equation. We describe the corresponding double Lie groups and we exhibit flat left invariant affine structures on them, by extending the flat affine structures on the oscillator group given in the previous section.\\
\\ 
Let $G$ be a real Lie group with Lie algebra $\mathfrak{g}$. A \textbf{\textit{solution of classical Yang-Baxter equation} (CYBE)} on $G,$ is a bivector $r\in\wedge^2\mathfrak{g}$ satisfying  
\begin{equation}\label{SYB}
[r,r]=0
\end{equation}
where
$[r,r]\in\wedge^3\mathfrak{g}$ is the Schouten bracket (see \cite{DZ}). The equation \eqref{SYB} is equivalent to
\begin{align*}
\alpha([r_{\sharp}(\beta),r_{\sharp}(\gamma)])+\beta([r_{\sharp}(\gamma),r_{\sharp}(\alpha)])+\gamma([r_{\sharp}(\alpha),r_{\sharp}(\beta)])=0,
\end{align*}
where $r_{\sharp}:\mathfrak{g}^*\rightarrow\mathfrak{g}$ is the linear homomorphism given by $\beta(r_{\sharp}(\alpha))=r(\alpha,\beta).$\\
Every solution $r$ of  CYBE on $G$, determites  a Lie algebra structure on the dual vector space  $\mathfrak{g}^*$ (see \cite{KS}), with Lie bracket  given by 
\begin{equation*}
[\alpha,\beta]_r=\mathsf{ad}^*_{r_{\sharp}(\beta)}\alpha-\mathsf{ad}^*_{r_{\sharp}(\alpha)}\beta
\end{equation*}
\begin{definition}
	The Lie algebra $\mathfrak{g}^*(r):=(\mathfrak{g}^*,[\ ,\ ]_r)$ is called dual Lie algebra of $\mathfrak{g}$ relative to $r.$ We denote by $G^*(r)$ a Lie group with Lie algebra  $\mathfrak{g}^*(r).$
\end{definition}
Notice that, if $\langle\ ,\ \rangle$ is an orthogonal structure on $\mathfrak{g},$ then $\mathfrak{g}^*(r)$ has a scalar product given by 
\begin{equation}
\langle\alpha,\beta\rangle^*=\langle\phi^{-1}(\alpha),\phi^{-1}(\beta)\rangle\label{estructuraortogonal}
\end{equation}
where $\phi(x)=\langle x,\ \rangle.$ We denote by $k^*$ the riemannian pseudo-metric on $G^*(r)$ determined by $\langle\quad,\quad\rangle^*$ and  $\nabla^*$ the Levi-Civita connection associated to $k^*$.\\
The essential tool to be used in this section is provided by the following result (see \cite{Med4}). 

\begin{theorem}[Boucetta-Medina]\label{MB}
	Let $(G,k)$ be a Lie group endowed with a bi-invariant pseudo-Riemannian metric and let $r$ be a solution of the generalized classical Yang-Baxter equation on $G$. Then, we have the following
	\begin{enumerate}
		\item $(G^*(r),k^*)$ is a locally symmetric pseudo-Riemannian manifold, i.e., $$\nabla^* R=0,$$
		where $R$ is the curvature of $k^*.$ In particular, $R$ vanishes identically when $r$ is a solution of the CYBE.
		\item If $k^*$ is flat then it is complete if and only if $G^*(r)$ is unimodular.  
	\end{enumerate} 
	If the two conditions are verified then $G^*(r)$ is solvable.  
\end{theorem}

Let $r$ be a solution of the CYBE on $G,$ then every Lie group of $G^*(r)$ is endowed with a flat affine left invariant connection given by 
\begin{equation*}
\alpha\cdot\beta:=-\mathsf{ad}^*_{r_{\sharp}(\alpha)}\beta
\end{equation*} 
for $\alpha,\beta\in\mathfrak{g}^*$ (see for instance \cite{DM}).  
The real vector space $\mathcal{D}(\mathfrak{g},r):=\mathfrak{g}\oplus\mathfrak{g}^*(r)$ endowed with the Lie bracket
$$[(x,\alpha),(y,\beta)]=([x,y]+\mathsf{ad}_{\alpha}^*y-\mathsf{ad}_{\beta}^*x,[\alpha,\beta]_r+\mathsf{ad}_x^*\beta-\mathsf{ad}_y^*\alpha)$$ 
is called \textbf{the double Lie algebra of Lie algebra $\mathfrak{g}$ relative to $r$.} Moreover, it is verified that the scalar product  $$\langle(x,\alpha),(y,\beta)\rangle:=\beta(x)+\alpha(y)$$
turns  $\mathcal{D}(\mathfrak{g},r)$ into an orthogonal Lie algebra. It is verified that $\mathfrak{g}$ y $\mathfrak{g}^*(r)$ are totally isotropic subalgebras of  $\mathcal{D}(\mathfrak{g},r)$  (see \cite{D}). Therefore, if  $\dim\mathfrak{g}=n,$ it holds that $\langle\ ,\ \rangle$ has signature $(n,n).$\\  

Boucetta M. and Medina A. in \cite{Med4} determined generic solutions of the CYBE of the $\lambda$-oscillator group $O_{\lambda}$ of dimension $n$. In the case  $n=4$ such solutions of the CYBE are written as 
$$r=\alpha_1e_0\wedge e_1+\alpha_2e_0\wedge e_2+\alpha_3e_0\wedge d$$
for $\alpha_1,\alpha_2,\alpha_3\in\R.$ The dual Lie algebra  $\mathfrak{o}^*(r)=\mathsf{Vect}\{e_0^*,e_1^*,e_2^*,d^*\}$ has non zero Lie brackets (except antisymmetry):
\begin{align*}
[e_0^*,e_1^*]=-\alpha_2d^*+\alpha_3e_2^* && [e_0^*,e_2^*]=\alpha_1d^*-\alpha_3e_1^*.
\end{align*}
The left symmetric product compatible with the Lie bracket is given by the following products 
\begin{align}\label{productoYBC}
\begin{split}
e_0^*\cdot e_0^*=\alpha_2e_1^*-\alpha_1e_2^*,&\quad e_0^*\cdot e_1^*=-\alpha_2d^*+\alpha_3e_2^*,\\
e_0^*\cdot e_2^*=&\alpha_1d^*-\alpha_3e_1^*.
\end{split}
\end{align}
The following proposition describes the dual Lie groups of $O$ relative to  solutions of the CYBE. 
\begin{lemma} The simply connected flat affine Lie groups  $O^*(r)$ corresponding to the Lie algebras $\mathfrak{0}^*(r)$ where $r=\alpha_1e_0\wedge e_1+\alpha_2e_0\wedge e_2+\alpha_3e_0\wedge d$ is a solution of the CYBE on $O$, are given by
	\begin{enumerate}
		\item $O^*(r)=\R^4,$ if $r=0.$
		\item $O^*(r)=\mathsf{H}_3\times\R\quad \mbox{(direct product)},$
		if $r=\alpha_1e_0\wedge e_1\neq0$ o $r=\alpha_2e_0\wedge e_2\neq0.$  
		\item $O^*(r)=\R^4$ with product
		$$(x,y,z,w)(x',y',z',w')=(x+x'+(w-z)y'+wz'+zw',y+y',z+z',w+w'),$$
		if $\alpha_3=0$ and $\alpha_1\alpha_2\neq0.$ 
		\item $O^*(r)=(\C\rtimes\R)\times\R$ with product $$(z,s,t)(z',s',t')=(z+e^{i\alpha_3s}z',s+s',t+t'),$$
		if $\alpha_3\neq0.$
	\end{enumerate}
\end{lemma}
\begin{proof}
When $\alpha_1,\alpha_2,\alpha_3=0,$ we obtain that $\mathfrak{o}^*(r)$ is abelian, then $O^*(r)=\R^4$.\\
In the second case, the Lie algebra $\mathfrak{o}_4^*(r)$ is isomorphic to  direct sum $\mathfrak{h}_3\oplus\R d,$ it follows  that $O^*(r)=\mathsf{H}_3\times\R.$\\
If $\alpha_3=0,$ the Lie algebra $\mathfrak{o}^*(r)$ is isomorphic to Lie algebra $$\R d^*\times_{g}\mathfrak{m}$$ central extension of abelian algebra $\mathfrak{m}=\mathsf{Vect}_{\R}\{e_0^*,e_1^*,e_2^*\}$ by 
means of the scalar 2-cocicle $g$ of $\mathfrak{
	m}$ defined by
\begin{align}\label{coci1}
\begin{split}
g(e_0^*,e_1^*)&=-1,\quad g(e_0^*,e_2^*)=1,\\
g(e_i^*,e_j^*)&=0\quad\mbox{for $(i,j)\neq(0,1),(0,2).$}
\end{split}
\end{align}
Then, the Lie group $O^*(r)$ with Lie algebra $\mathfrak{o}^*(r)$ is $\R^4$ with product
$$(x,y,z,w)(x',y',z',w')=(x+x'+(w-z)y'+wz'+zw',y+y',z+z',w+w')$$
\\
For $\alpha_3\neq0,$ we obtain that $\mathfrak{o}^*(r)$ is isomorphic to Lie algebra $\mathsf{Vect}_{\R}\{f_0,f_1,f_2,d\}$ 
\begin{align*}
[d,f_1]=\alpha_3f_2,&&[d,f_2]=-\alpha_3f_1.
\end{align*}
It follows 
that $$\mathfrak{o}^*(r)=\mathsf{Vect}_{\R}\{f_1,f_2\}\rtimes_d\R d\oplus\R f_0$$ where $d\in\mathsf{End}(\mathsf{Vect}_{\R}\{f_1,f_2\}).$ It verifies that  $O^*(r)=(\C\rtimes\R)\times\R$ with the product \[(s,z,t)(s',z',t')=(s+s',z+e^{i\alpha_3s}z',t+t'),\text{ for all }(s,z,t),(s',z',t')\in O^*(r).\]
\end{proof}
\begin{theorem}
	The Lie groups $O^*(r)$ defined by each solution of the CYBE are endowed with a flat left invariant Lorentz metric, a symplectic form and a flat  affine left invariant connection.   
\end{theorem}
\begin{proof}
Theorem \eqref{MB} implies that $O^*(r)$ admits a flat left invariant  Lorentz metric. Moreover, this metric is given by the scalar product  on  $\mathfrak{o}^*(r)$ defined  by \eqref{estructuraortogonal} and the  left symmetric product is given by \eqref{productoYBC}. It can be verified that the corresponding connection is the Levi-Civita connection. 
\\ 
A symplectic form on Lie group $O^*(r)$ is given by the non degenerate scalar 2-cocycle of the algebra $\mathfrak{o}^*$ as follows
\begin{gather*}
\omega(e_1^*,e_2^*)=\omega(e_0^*,d^*)=1\quad\mbox{y}\quad \mathsf{Vect}_{\R}\{e_1^*,e_2^*\}\bot_{\omega}  \mathsf{Vect}_{\R}\{e_0^*,d^*\}
\end{gather*}
for $x,y,z\in\mathfrak{o}^*(r)$. The formula 
\begin{equation}\label{productosimplectico}
\omega(xy,z)=-\omega(y,[x,z])
\end{equation}
defines a left symmetric product compatible with the Lie bracket of  $\mathfrak{o}^*(r).$\\
The non zero products defined by \eqref{productosimplectico} are given by  
\begin{align*}
e_0^*e_2^*&=-\alpha_3e_1^*,& e_1^*e_1^*&=-\alpha_3d^*,&
e_2^*e_2^*&=-\alpha_3d^*,\\
e^*_0e^*_0&=-\alpha_2e_2^*-\alpha_1e_0^*,& e_0^*e_1^*&=\alpha_3e_2^*.
\end{align*}
The theorem follows from the above statements.
\end{proof}
\begin{remark}
	The left symmetric product on $\mathfrak{o}^*(r)$ given in \eqref{productosimplectico} defines a complete flat affine left invariant  connection on $O^*(r)$ which is not compatible with the symplectic form $\omega$.
\end{remark}
\begin{remark}[Double Lie groups of the oscillator]
	The double Lie algebra of $\mathfrak{o}$ relative to a solution $r$ of the CYBE given above, correspond to the vector space $$\mathcal{D}(\mathfrak{o},r)=\mathfrak{o}^*\oplus\mathfrak{o}=\mathsf{Vect}_{\R}\{e_0,e_1.e_2,d,e_0^*,e_1^*,e_2^*,d^*\}$$  
	with Lie brackets (non zero and except antisymmetries) 	
	\begin{align*}
	[e_1,e^*_0]&=-\alpha_3e_2-e^*_2,&[e_2,e_0^*]&=\alpha_3e_1+e_1^*,&[d,e_0^*]&=-\alpha_2e_1+\alpha_1e_2,\\
	[e_2,e_1^*]&=-\alpha_3e_0-d^*,&[d,e_1^*]&=\alpha_2e_0-e_2^*,&[e_1,e_2^*]&=\alpha_3e_0+d^*,\\
	[d,e_2^*]&=-\alpha_1e_0-e_1^*,&[e_0^*,e_1^*]&=-\alpha_2d^*+\alpha_3e_2^*,&[e_0^*,e_2^*]&=\alpha_1d^*-\alpha_3e_1^*,\\
	[e_1,e_2]&=e_0,&[d,e_1]&=e_2,&[d,e_2]&=-e_1.
	\end{align*}
\end{remark}

\section{Development of flat affine manifold}
This section is devoted to the description of the developing map  corresponding to each flat affine left invariant  structure on $\mathsf{Aff}(\R)$ found in \cite{MS} and to each flat affine structure on the oscillator group $O$   exhibited in Section 2.
The method we will apply has been introduced by J.L. Koszul in \cite{K2}.

\begin{proposition}\label{desarrollante}
Let $(M,\nabla)$ be a simply connected flat affine manifold  of dimension $n$. Then there exist an affine immersion  $D:(M,\nabla)\rightarrow(\R^n,\nabla^0)$ and a Lie group  homomorphism $A:\mathsf{Aff}(M,\nabla)\rightarrow\mathsf{Aff}(\R^n,\nabla^0)$ such that the following diagram commutes
	\begin{equation}\label{diagrama}
		\begin{tikzcd}
			M\arrow{r}{D}\arrow{d}{F}&\R^n\arrow{d}{A(F)}\\
			M\arrow{r}{D}&\R^n
		\end{tikzcd}
	\end{equation}
	Moreover, $A$ is an immersion.  
\end{proposition}
\begin{proof}
Let $p\in M$ and $v\in T_p(M).$
For a curve $\alpha$ of class $C^1$ on $M$ such that $\alpha(0)=p$ there exist a unique smooth vector field $X(\alpha)$ along $\alpha$ that is  parallel relative to $\nabla$ such that $X(\alpha)_p=v.$ Since $M$ is connected and it has an affine atlas, i.e., the transition maps are local affine transformations of open sets of $\R^n,$ we can extend $X(\alpha)$ to a smooth vector field $X^v$ of $M$ such that $X_p^v=v,$ and $\nabla X^v=0.$ 
\\
Let $\eta$ be a differential  1-form on $M$ with values in $T_p(M)$ such that $\eta(X^v)=v.$ We shall see that $\eta$ is closed. Notice that 
	\begin{align*}
	(\nabla\eta)(X^v,Y)&=Y(\eta(X^v))-\eta(\nabla_YX^v)=0
	\end{align*}
for all $Y\in\mathfrak{X}(M).$ It follows that $\nabla\eta=0$ and therefore 
	\begin{align*}
	0&=(\nabla\eta)(X,Y)-(\nabla\eta)(Y,X)\\
	&=Y(\eta(X))-\eta(\nabla_YX)-X(\eta(Y))+\eta(\nabla_XY)\\
	&=Y(\eta(X))-X(\eta(Y))-\eta(\nabla_YX-\nabla_XY)\\
	&=Y(\eta(X))-X(\eta(Y))-\eta([Y,X])\\
	&=2\di\eta(Y,X)
	\end{align*}
for all $X,Y\in\mathfrak{X}(M).$ \\
Since $M$ is simply connected, there exist a unique differential map  $D$ from $M$ in $T_p(M)$ such that $\di D=\eta.$ In fact, taking  a smooth curve $\tau$ in $M$ from $p$ to $q,$ it follows from Stokes's theorem that $\int_{\tau}\eta$ does not depend of the curve. Hence the following vector is well defined
	\begin{equation}\label{DA}
	D(x)=\int_p^x\eta.
	\end{equation}
Now we will show that $D$ is an affine immersion from $M$ in $T_p(M)$.  Let $\sigma(t)$ be a smooth curve on $M,$ since
	\begin{align*}
	\dot{\sigma}(t)(\eta(\dot{\sigma}(t)))=\eta(\nabla_{\dot{\sigma(t)}}\dot{\sigma}(t)),
	\end{align*} 
it follows that $\sigma$ is a geodesic if and only if $\eta(\dot{\sigma}(t))$ is constant. Then  the geodesics of $M$ are curves $\sigma$ such that the map $\eta\circ\sigma$ is constant.\\
For a geodesic $\sigma(t)$ on $(M,\nabla)$ we have the following  
	\[
	D(\sigma(t))=\int_p^{\sigma(0)} \eta+\int_{\sigma(0)}^{\sigma(t)} \eta=\int_p^{\sigma(0)} \eta+\int_{\sigma(0)}^t \eta(\dot{\sigma}(t))\di t=\eta(\dot{\sigma}(t))t+\int_p^{\sigma(0)} \eta
	\]
Hence, $D(\sigma(t))$ is a line on $T_p(M),$ i.e., $D:(M,\nabla)\rightarrow (T_pM,\nabla^0)$ is an affine map. Since $\di D=\eta,$ the rank of  $D$ is constant and equal to  $n=\dim M$ in every point, thats means $D$ is an immersion.\\
Next we will prove that  the following map is a Lie group homomorphism 
	\begin{align*}
	\begin{array}{rccl}
	A:&\mathsf{Aff}(M,\nabla)&\rightarrow&\mathsf{Aff}(\R^n,\nabla^0)\\
	&F&\mapsto&(Q_F,L_F)
	\end{array}
	\end{align*}
First we will see that the map is well defined. Let $L_F(v)=u=F_{\ast,F^{-1}(p)}\left(X^v_{F^{-1}(p)}\right)$, with  $v\in T_p(M)$, and $Q_F=D(F(p))$.
Since $\nabla F_{\ast}(X^v)=0$ we have that  $F_{\ast}(X^v)=X^u,$ which means that $L_F$ is $\R$-linear.\\
Now, we will show that $A(F\circ G)=A(F)\circ A(G)$ for all $F,G\in\mathsf{Aff}(M,\nabla).$\\
It directly follows from definition that $L_{F\circ G}=L_F\circ L_G$. Moreover 
	\begin{align}
	\begin{split}\label{integralconmuta1}
	Q_{F\circ G}&=D(F(G(p)))=\int_{\tau}\eta,\quad\mbox{where $\tau$ is a curve from $p$ to $F(G(p))$}\\
	&=\int_{\alpha}\eta+\int_{\alpha'}\eta,\quad\mbox{with $\alpha$ a curve from $p$ to $F(p)$}\\
	&\hspace{3.5 cm}\mbox{and $\alpha'$ is a curve from $F(p)$ to $F(G(p))$}\\
	&=D(F(p))+\int_{F\circ\beta}\eta,\quad\mbox{ $\beta$ is a curve from $p$ to $G(p))$}\\
	&=Q_F+\int_{\beta}F^{\ast}\eta
	\end{split}
	\end{align}
	On the other hand,
	\begin{align}
	\begin{split}\label{integralconmuta2}
	(L_F\circ\eta)_p(X^v_p)&=L_F(\eta_p(X^v_p))\\
	&=L_F(v)\\
	&=u\\
	&=\eta_{F(p)}(X^u_p)\\
	&=\eta_{F(p)}(F_{\ast,p}(X^v_p))=(F^{\ast}\eta)_p(X^v_p)
	\end{split}
	\end{align}
	It follows that   $\int_{\beta}F^{\ast}\eta=L_F(Q_G),$ hence, from   \eqref{integralconmuta1} we get  that $A$ is a group  homomorphism.\\ 
	Now notice that for $q\in M$ we have
	\begin{align*}
	D(F(q))&=\int_{\tau}\omega,\quad\mbox{where $\tau$ is a curve from $p$ to $F(q)$}\\
	&=Q_F+L_F(D(q))\\
	&=A(F)(D(q)),
	\end{align*}
	that is,  \eqref{diagrama} is a commutative diagram.\\
	Finally we prove that $A$ is an immersion. First notice that the following diagram is commutative 
	\begin{equation}
	\begin{tikzcd}
	\mathsf{Aff}(M,\nabla)\arrow{r}{A}\arrow{d}{\Phi}&\mathsf{Aff}(\R^n,\nabla^0)\arrow{d}{\Psi}\\
	L(M)\arrow{r}{\widetilde{D}}&L(\R^n)
	\end{tikzcd}
	\end{equation} 
where $\Phi$ maps an affine transformation $F$ of $(M,\nabla)$ in  the linear frame $\widetilde{F}(u_0)$ with $u_0$ a linear frame of $M$ at $p$ fixed and $\Psi$ maps the affine transformation $T$ of $(\R^n,\nabla^0)$ in the linear frame $\widetilde{T}(u_1)$ with $u_1=\widetilde{D}(u_0)$ linear frame of $\R^n$ at $0.$ Since   $\Phi$ and $\Psi$ are embeddings (see page 42 \cite{K}), they are, in particular, immersions. Moreover as $D$ is immersion its natural lift is also an immersion. Therefore $A$ is an immersion.
\end{proof}
The map $D:M\rightarrow T_p(M)$ defined in the previous proposition is called \textbf{\textit{developing map}} of $(M,\nabla)$ at $p$.\\
\\
For the following proposition we will consider $M=\mathsf{Aff}(\R)_0$  with the flat affine structures $\nabla_2(\alpha)$ and $\nabla_1(\alpha)$ respectively given by  $\mathcal{F}_2(\alpha)$ and $\mathcal{F}_1(\alpha)$ (see \cite{MS}).

\begin{proposition}
The corresponding developing maps of $\mathsf{Aff}(\R)_0$ with values in  $\mathsf{aff}(\R)$ relative to $\nabla_2(\alpha)$ and $\nabla_1(\alpha)$ are respectively given by
	\begin{align*}
		D_{2,\alpha}(x,y)&=\left\{\begin{array}{ll}
			\left(\frac{1}{\alpha}(x^{\alpha}-1)\ ,\ x^{\alpha}(y+1)-\frac{1}{\alpha+1}(\alpha x^{\alpha+1}+1)\right)& \mbox{if $\alpha\neq0,-1$}\\
			\\
			\left(1-\frac{1}{x}\ ,\ \frac{1+y}{x}-1+\ln x\right)& \mbox{if $\alpha=-1$}\\
			\\
			(\ln(x)\ ,\ y)& \mbox{if $\alpha=0$}
		\end{array}\right.
	\end{align*}
	\begin{align*}
		D_{1,\alpha}(x,y)&=\left\{\begin{array}{ll}
			\left(\frac{1}{\alpha}(x^{\alpha}-1)\ ,\ y+\frac{1}{\alpha-1}(x^{\alpha}-\alpha x)+1\right)& \mbox{if $\alpha\neq0,1$}\\
			\\
			\left(x-1\ ,\ y+1+x(\ln x-1)\right)& \mbox{if $\alpha=1$}
		\end{array}\right.
	\end{align*}
\end{proposition}
\begin{proof}
\begin{enumerate}
	\item 
	Let $X$ be a smooth vector field on $\mathsf{Aff}(\R)_0$ such that  $\nabla_2(\alpha)X=0.$   The following identities can be verified
	\begin{gather}
		\nabla_2(\alpha)_{e_1^+} X=0\label{ecudif1}\\
		\nabla_2(\alpha)_{e_2^+} X=0\label{ecudif2}
	\end{gather}
	Solving \eqref{ecudif1} and \eqref{ecudif2} with initial conditions $X_{(1,0)}=(a,b)\in\mathfrak{g},$ we have 
	\begin{align*}
		X=ax^{1-\alpha}\dx+(b+\alpha a(x-y-1))x^{-\alpha}\dy
	\end{align*}
	It is easy to verify that the closed 1-form $\eta$ on $\mathsf{Aff}(\R)_0$ with values in $\mathfrak{g}$ such that $\eta(X)=(a,b),$ is given by
	\begin{align*}
		\eta=\left(x^{\alpha-1}\di x,\ \alpha(y-x+1)x^{\alpha-1}\di x+x^{\alpha}\di y \right)
	\end{align*}
	Taking de curve $\tau(t)=(1+(x-1)t, yt),$ with $t\in[0,1],$ the  developing map is written as follows
	\begin{enumerate}
		\item if $\alpha\neq0,-1.$
		\begin{align*}
			&D_{2,\alpha}(x,y)=\int_{\tau}\eta\\
			&=\left(\int_0^1(1+(x-1)t)^{\alpha-1}(x-1)\dt, \int_0^1\alpha(x-1)(y-x+1)t(1+(x-1)t)^{\alpha-1}\right.\\
			&\hspace{9.2cm}+(1+(x-1)t))^{\alpha}y \ \dt\bigg)\\
			&=\left(\frac{1}{\alpha}(x^{\alpha}-1),\ x^{\alpha}(y+1)-\frac{1}{\alpha+1}(\alpha x^{\alpha+1}+1)\right)
		\end{align*}
		\item For case $\alpha=-1:$ 
		\begin{align*}
			D_{2,-1}(x,y)&=\int_{\tau}\eta\\
			&=\left(\int_0^1\frac{x-1}{(1+(x-1)t)^2}\dt,\int_0^1\frac{(x-y-1)(x-1)t}{(1+(x-1)t)^2}\dt+\frac{y}{1+(x-1)t}\dt\right)\\
			&=\left(1-\frac{1}{x},\frac{1+y}{x}-1+\ln x\right)
		\end{align*}
		\item  when $\alpha=0:$ 
		\begin{align*}
			D_{2,0}(x,y)&=(\ln(x)\ ,\ y)
		\end{align*}
	\end{enumerate}
	\item If $X\in\mathfrak{X}(\mathsf{Aff}(\R)_0)$ such that $X_{(1,0)}=(a,b)\in\mathfrak{g},$ and
	\begin{gather*}
		\nabla_1(\alpha)_{e_1^+}X=0\\
		\nabla_1(\alpha)_{e_2^+}X=0
	\end{gather*}
	it can be verified that $X$ is as follows 
	\begin{align*}
		X=\left\{
		\begin{array}{ll}
			ax^{1-\alpha}\dx+\left(b+\frac{a\alpha}{\alpha-1}(x^{1-\alpha}-1)\right)\dy&\mbox{si $\alpha\neq1$}\\
			\\
			a\dx+(b-a\ln x)\dy&\mbox{si $\alpha=1$}
		\end{array}
		\right.
	\end{align*}
	Hence the 1-form $\eta$ on $\mathsf{Aff}(\R)_0$ verifying $\eta(X)=(a,b),$ is given by
	\begin{align*}
		\eta=
		\left\{
		\begin{array}{ll}
			\left(x^{\alpha-1}\di x\ ,\ \frac{\alpha}{\alpha-1}(x^{\alpha-1}-1)\di x+\di y\right)&\mbox{if $\alpha\neq1$}\\
			\\
			(\di x\ ,\ \ln x\di x+\di y)&\mbox{if $\alpha=1$}
		\end{array}
		\right.
	\end{align*}
	Therefore, the developing map is as follows  
	\begin{align*}
		D_{1,\alpha}(x,y)&=\int_{\tau}\eta\\
		&=\left\{\begin{array}{ll}
			\left(\frac{1}{\alpha}(x^{\alpha}-1)\ ,\ y+\frac{1}{\alpha-1}(x^{\alpha}-\alpha x)+1\right)& \mbox{if $\alpha\neq0,1$}\\
			\\
			\left(x-1\ ,\ y+1+x(\ln x-1)\right)& \mbox{if $\alpha=1$}
		\end{array}\right.
	\end{align*}
\end{enumerate}
\end{proof}


  From now on we will denote by $\nabla_1(t,s),$ $\nabla_2(\alpha,t),$  $\nabla_3$ and $\nabla_4$  the flat affine connections on $O$ respectively given by the products
$\mathcal{F}_1(0,t,s,1/2),$ $\mathcal{F}_2(\alpha,t,1/2),$ $\mathcal{F}_3(1/2)$ y $\mathcal{F}_4(1/2)$.\\

\begin{proposition}
The corresponding developing maps of oscillator $O$ relative to the flat affine connections $\nabla_1(t,s),$ $\nabla_2(\alpha,t),$  $\nabla_3$ and $\nabla_4$ are  respectively 
	\begin{align*}
		D_{1,t,s}(x_1,x_2,x_3,x_4)&=\left(x_1+\frac{t}{2}(x_2^2+x_3^2)+\frac{s}{2}x_4^2,\ x_2,\ x_3,\ x_4 \right)\\
		D_{2,\alpha,t}(x_1,x_2,x_3,x_4)&=\left(x_1+\frac{t}{2}(x_2^2+x_3^2),\ x_2,\ x_3,\ \frac{1}{\alpha}(e^{\alpha x_4}-1) \right)\\
		D_3(x_1,x_2,x_3,x_4)&=\left(x_1,\ x_2,\ x_3,\  x_4+\frac{1}{2}(x_2^2+x_3^2) \right)\\
		D_4&=(F_1,F_2,F_3,F_4)
	\end{align*}
	where $F_i,\ i=1,2,3,4$ are detailed at the end of the proof.
\end{proposition}
\begin{proof}
\begin{enumerate}
	\item Let $X\in\mathfrak{X}(O)$ such that $X_{(0,0,0,0)}=(a_1,a_2,a_3,a_4)$ with
	\begin{align*}
		\nabla_2(\alpha,t)_{e_i^+}X&=0,\quad\mbox{for $i=0,1,2$}\\
		\nabla_2(\alpha,t)_{d^+}X&=0,
	\end{align*}
	where $(e_0^+,e_1^+,e_2^+,d^+)$ is the basis of left invariant vector fields on $O$ defined in \eqref{Eq:basisofleftinvariant}.\\
	These equations define a  system of partial differential equations  whose solutions determine the vector field 
	$$X=(a_1-a_2tx_2-a_3tx_3)\dxu+a_2\dxd+a_3\dxt+a_4e^{-\alpha x_4}\dxc.$$
	It is easy to verify that the 1-form $\eta$ such that $\eta(X)=a$ with $a$ in $\mathfrak{o},$ is written
	$$\eta=(\di x_1+tx_2\di x_2+tx_3\di x_3,\di x_2,\di x_3,e^{\alpha x_4}\di x_4)$$
	Hence the  developing map is as follows
	\begin{align*}
		D_{2,\alpha,t}(x_1,x_2,x_3,x_4)&=\int_{\tau}\eta\\
		&=\left(\int_0^1x_1\di r+tx_2^2r\di r+tx_3^2r\di r,\int_0^1x_2\di r,\int_0^1x_3\di r,\int_0^1x_4e^{\alpha x_4r}\di r\right)\\
		&=\left(x_1+\frac{t}{2}(x_2^2+x_3^2),\ x_2,\ x_3,\ \frac{1}{\alpha}(e^{\alpha x_4}-1) \right)
	\end{align*}
	\item The smooth vector field $X$ on $O$ such that $X_{(0,0,0,0)}=(a_1,a_2,a_3,a_4)$ with $\nabla_1(t,s)X=0,$ can be written as
	$$X=(a_1-a_2tx_2-a_3tx_3-a_4sx_4)\dxu+a_2\dxd+a_3\dxt+a_4\dxc.$$
	On the other hand, the 1-form on $O$ such that $\eta(X)=(a_1,a_2,a_3,a_4)\in\mathfrak{o},$  is given by
	$$\eta=(\di x_1+tx_2\di x_2+tx_3\di x_3+sx_4\di x_4,\di x_2,\di x_3,\di x_4)$$  
	Taking a smooth curve $\tau$ on $O$ like $\tau(r)=(x_1r,x_2r,x_3r,x_4r)$ with $r\in[0,1],$ we obtain the  developing map 
	\begin{align*}
		D_{1,t,s}(x_1,x_2,x_3,x_4)&=\int_{\tau}\eta\\
		&=\left(\int_0^1x_1\di r+tx_2^2r\di r+tx_3^2r\di r+sx_4^2r\di r,\int_0^1x_2\di r,\int_0^1x_3\di r,\int_0^1x_4\di r\right)\\
		&=\left(x_1+\frac{t}{2}(x_2^2+x_3^2)+\frac{s}{2}x_4^2,\ x_2,\ x_3,\ x_4 \right)
	\end{align*}
	\item Let $X$ be a smooth vector field on $O$ satisfying  $X_{(0,0,0,0)}=(a_1,a_2,a_3,a_4)$ and 
	\begin{align*}
		\nabla_{3,e_i^+}X&=0,\quad\mbox{para $i=0,1,2$}\\
		\nabla_{3,d^+}X&=0
	\end{align*}
	whose solutions determine the vector field 
	$$X=a_1\dxu+a_2\dxd+a_3\dxt+(a_4-a_2x_2-a_3x_3)\dxc.$$
	The 1-form $\eta,$ which $\eta(X)=(a_1,a_2,a_3,a_4)\in\mathfrak{o},$ is written
	$$\eta=(\di x_1,\di x_2,\di x_3,x_2\di x_2+x_3\di x_3+\di x_4)$$  
	As a consequence we get that the  developing map of $(O,\nabla_3)$ is given by
	\begin{align*}
		D_3(x_1,x_2,x_3,x_4)&=\int_{\tau}\eta\\
		&=\left(\int_0^1x_1\di r,\int_0^1x_2\di r,\int_0^1x_3\di r,\int_0^1x_2^2r +x_3^2r+x_4\di r\right)\\
		&=\left(x_1,\ x_2,\ x_3,\  x_4+\frac{1}{2}(x_2^2+x_3^2) \right)
	\end{align*}
	\item Finally, we calculate the developing map for $(O,\nabla_4).$ A smooth vector field $X$ on $O$ that is parallel relative to $\nabla_4$  with $X_{(0,0,0,0)}=(a_1,a_2,a_3,a_4)$ has the form
	\begin{align*}
		X=&\left(-\frac{\hat{a}_9}{2}\sin(x_4)+\frac{\hat{a}_1}{2}\cos(x_4)-\frac{\hat{a}_2}{2} x_2-\frac{\hat{a}_3}{2}x_3+\hat{a}_4\right)e_0^+\\
		&+\left(\hat{a}_5\sin(x_4)+\hat{a}_6\cos(x_4)+\frac{\hat{a}_3}{2}x_2-\frac{\hat{a}_2}{2}x_3+\frac{\hat{a}_4}{2}\right)e_1^+\\
		&+(\hat{a}_7\sin(x_4)+\hat{a}_8\cos(x_4)+\hat{a}_3x_2-\hat{a}_2x_3+\hat{a}_4)e_2^+\\
		&+\left(\hat{a}_1\sin(x_4)+\hat{a}_9\cos(x_4)-\hat{a}_3x_2+\hat{a}_2x_3-\hat{a}_4\right)d^+
	\end{align*}
	where 
	\begin{align*}
		\hat{a}_1&=\frac{a_2}{2}+a_3+\frac{5a_4}{4}&
		\hat{a}_2&=a_3+a_4\\
		\hat{a}_3&=-a_2-\frac{a_4}{2}&
		\hat{a}_4&=a_1+\frac{a_3}{2}+\frac{a_4}{2}\\
		\hat{a}_5&=\frac{1}{2}\left(-\frac{a_2}{2}+a_3+\frac{3a_4}{4}\right)&
		\hat{a}_6&=\frac{1}{2}\left(a_2+\frac{a_3}{2}+a_4\right)\\
		\hat{a}_7&=-\frac{3a_2}{2}-a_3-\frac{7a_4}{4}&
		\hat{a}_8&=-a_2+\frac{3a_3}{2}+a_4\\
		\hat{a}_9&=a_2-\frac{a_3}{2}
	\end{align*}
	On the other hand, the 1-form on $O$ with values in $\mathfrak{o}$ such that $\eta(X)=(a_1,a_2,a_3,a_4)$ is $\eta=(\omega_1,\ \omega_2,\ \omega_3,\ \omega_4)$ where
	\begin{align*}
		\omega_1&=\di x_1+\frac{1}{2}\left((1-x_3)\cos(x_4)+(1+x_2)\sin(x_4)-\frac{1}{2}\right)\di x_2\\
		&+\frac{1}{2}((1+x_2)\cos(x_4)+(x_3-1)\sin(x_4)-1)\di x_3\\
		&+\frac{1}{2}\left(\left(\frac{3}{2}+x_2-\frac{x_3}{2}\right)\cos(x_4)+\left(\frac{x_2}{2}+x_3-\frac{1}{2}\right)\sin(x_4)-\frac{5}{4}\right)\di x_4\ ,
	\end{align*}
	\begin{align*}
		\omega_2&=\frac{1}{2}\left((1+x_2)\cos(x_4)+\left(\frac{1}{2}+x_3\right)\sin(x_4)+1\right)\di x_2\\
		&+\frac{1}{2}\left(\left(\frac{1}{2}+x_3\right)\cos(x_4)-(x_2+1)\sin(x_4)-\frac{1}{2}\right)\di x_3\\
		&+\frac{1}{2}\left(\left(1+\frac{x_2}{2}+x_3\right)\cos(x_4)+\left(x_3-x_2-\frac{3}{4}\right)\sin(x_4)-1\right)\di x_4\ ,
	\end{align*}
	\begin{align*}
		\omega_3&=\left((x_2-1)\cos(x_4)+\left(\frac{3}{2}+x_3\right)\sin(x_4)+1\right)\di x_2\\
		&+\left(\left(\frac{3}{2}+x_3\right)\cos(x_4)+(1-x_2)\sin(x_4)-\frac{1}{2}\right)\di x_3\\
		&+\left(\left(1+\frac{x_2}{2}+x_3\right)\cos(x_4)+\left(\frac{x_3}{2}-x_2+\frac{7}{4}\right)\sin(x_4)-1\right)\di x_4\ ,
	\end{align*}
	\begin{align*}
		\omega_4&=\left((1-x_2)\cos(x_4)-\left(\frac{1}{2}+x_3\right)\sin(x_4)-1\right)\di x_2\\
		&+\left(\frac{1}{2}-\left(\frac{1}{2}+x_3\right)\cos(x_4)+(x_2-1)\sin(x_4)\right)\di x_3\\
		&+\left(1-\left(\frac{x_2}{2}+x_3\right)\cos(x_4)+\left(x_2-\frac{x_3}{2}-\frac{5}{4}\right)\sin(x_4)\right)\di x_4
	\end{align*}
	Therefore, the corresponding developing map has the following  components
	\begin{align*}
		F_1(x_1,x_2,x_3,x_4)&=x_1-\frac{1}{4} (x_2+5 x_4+2)-\frac{1}{4x_4}(2x_3 x_4+x_3)\\&+\frac{\sin (x_4)}{4x_4^2}
		\left(2x_2^2+2x_2 x_4 (x_4+2)+2x_3^2-x_3 (x_4-6) x_4+3 x_4^2\right)\\
		&-\frac{ \cos (x_4) }{4x_4}\left(2x_2^2+2 x_2 x_4+(2 x_3-1) (x_3+2 x_4)\right)
	\end{align*}
\end{enumerate}
\begin{align*}
	F_2(x_1,x_2,x_3,x_4)&=-\frac{x_4}{2}+\frac{\left(x_2^2+x_3^2\right)}{2 x_4^2} (\cos (x_4)-1)+\\
	&\frac{1}{8 }(4 x_2-2 x_3+2 (x_2+2 x_3+2) \sin (x_4)\\
	&+(4 x_2-4 x_3+3) \cos (x_4)-3)\\
	&+\frac{1}{4 x_4} \left(\left(2 x_2^2+x_3 (2 x_3+3)\right) \sin (x_4)+4 x_3 (\cos (x_4)-1)\right)
\end{align*}
\begin{align*}
	F_3(x_1,x_2,x_3,x_4)&=\left(\frac{x_2^2+x_3^2}{4 x_4^2}+\frac{4x_2-2x_3-7}{4} -\frac{x_2}{x_4}\right)\cos (x_4) \\
	&+ \left(2 x_2^2+x_2 (x_4-4)+2 (x_3 (x_3+x_4+2)+x_4)\right)\frac{\sin (x_4)}{2x_4}\\
	&-\frac{x_2^2+x_3^2}{ x_4^2}+\frac{7-2 x_3}{4}+\frac{x_2}{x_4}+x_2- x_4
\end{align*}
\begin{align*}
	F_4(x_1,x_2,x_3,x_4)&=\left(\frac{2 x_3+5}{4}- \frac{x_2^2+x_3^2}{x_4^2}-x_2\right)\cos (x_4)\\
	&-\left(\frac{x_2^2+x_3^2}{x_4}-\frac{2x_2-x_3}{x_4}+\frac{x_2}{2}+x_3\right)\sin (x_4)\\
	&+\frac{x_2^2+x_3^2}{ x_4^2}+\frac{2 x_3-5}{4}+x_4-x_2
\end{align*}
\end{proof}
\section{Symplectic invariant geometry on the Lie group  $\mathsf{Aff}(\R)$}
In what follows we consider the group of affine transformations of real line $G:=\mathsf{Aff}(\R)$ as a symplectic Lie group. On the natural coordinates of $G$ the left invariant symplectic form it is written as $\omega^+=-\frac{1}{x^2}\di x\wedge\di y.$\\
 
\begin{proposition}[\cite{MS}] There exits  two left invariant symplectic flat affine, non isomorphic, connections $\nabla_1$ and $\nabla_2$ on  $\mathsf{Aff}(\mathbb{R})$. These connections correspond to the left symmetric products on $\mathfrak{g}=$Lie$\mathsf{Aff}(\mathbb{R})$  given by the following tables
	\begin{align}
	&\begin{array}{c|cc}
	&e_1 &e_2\\
	\hline
	e_1&-e_1&e_2\\
	&\\
	e_2&0&0
	\end{array}\label{iso1}\\
	&\begin{array}{c|cc}
	&e_1 &e_2\\
	\hline
	e_1&-\frac{1}{2}e_1&\frac{1}{2}e_2\\
	&\\
	e_2&-\frac{1}{2}e_2&0
	\end{array}\label{iso2}
	\end{align}
\end{proposition}

Recall that a \textbf{\textit{symplectomorphism}} of $(M,\omega)$ is a diffeomorphism $f$ of $M$ such that $f^*\omega=\omega,$ i.e., a diffeomorphism which preserves the symplectic form $\omega.$  \\
An easy calculation shows the following
\begin{lemma}
A diffeomorphism $f=(f_1,f_2)$ of the manifold $\mathsf{Aff}(\mathbb{R})_0$ is a symplectomorphism relative to  $\omega^+$ if and only if it verifies the differential equation 
	$$\frac{\partial f_1}{\partial x}\frac{\partial f_2}{\partial y}-\frac{\partial f_1}{\partial y}\frac{\partial f_2}{\partial x}=\left(\frac{f_1}{x}\right)^2$$
\end{lemma}
In particular the following maps are (functionally independent) symplectomorphism 
\begin{align}
f(x,y)&=\left(\frac{axe^y}{1+axe^yp(y)}\ ,\ ae^y\right),\ \mbox{with $a>0,$ $p(y)\in C^{\infty}$ and $p>0$.}\label{simple1}\\
g(x,y)&=\left(e^y\ ,\ \frac{e^y}{x}\right).\label{simple2}\\
h(x,y)&=\left(p(x)\ , \ \frac{(h(x)^2)y}{h'(x)x^2}+l(x) \right),\ \mbox{with $p(x),l(x)\in C^{\infty}$ and $p>0$.}
\end{align}

 The previous lemma implies the following.
\begin{proposition} Let $\nabla_1$ and  $\nabla_2$ be the connections given by the products  \eqref{iso1} and \eqref{iso2}. Then the maps
	$$F(x,y)=(ax,-bx^{-1}+ay+d),$$
	and 
	$$G(x,y)=\left(ax,-2b\sqrt{x}+ay+d\right)$$
	where $a,b,d$ are real numbers with $a>0,$ are symplectomorphisms and   affine transformations  relatively  to $(G_0,\nabla_1)$ and respectively $(G_0,\nabla_2)$.  
\end{proposition}

We finish by computing the Hess connection associated to a   Lagrangian bi-foliation of $\mathsf{Aff}(\mathbb{R})$ and we show that each of these connections is flat. 

\begin{proposition}
There are on $(G,\omega^+)$ two transversal Lagrangian foliations given respectively by $\mathcal{F}_i=\R e_i^+,$ $i=1,2.$ The Hess connection associated to this bifolation  ($\mathcal{F}_1$ and $\mathcal{F}_2$)  is a flat affine and Lorentz connection. 
\end{proposition}
\proof
The formula $\omega(ab,c)=-\omega(b,[a,c])$ defines on the vector space $\mathfrak{g}=\mathsf{aff}(\R)$ a left symmetric product compatible with Lie bracket given by following table
\begin{align}\label{productoformulasym} 
\begin{array}{c|cc}
\odot&e_1&e_2\\
\hline
e_1&-e_1&0\\
e_2&-e_2&0
\end{array}
\end{align}
Clearly $\R e_1$ and $\R e_2$ are transversal Lagrangian Lie subalgebras of $(\mathfrak{g},\omega)$. 
Recall that the Hess connection associated to $\mathcal{F}_1$ and $\mathcal{F}_2$ (see \cite{Hess})  is defined by 
$$\nabla^H_{(X_1,X_2)}(Y_1,Y_2)=(X_1\odot Y_1+[X_2,Y_1]_1\ ,\ X_2\odot Y_2+[X_1,Y_2]_2)$$
where $[\ ,\ ]_i$ is the projection on the $i-$foliation and $\odot$ is the product given by  \eqref{productoformulasym}.\\
A direct computation shows that this connection correspond to bilinear product on $\mathfrak{g}$ given by
\begin{align}\label{ConexionHess} 
\begin{array}{c|cc}
\star&e_1&e_2\\
\hline
e_1&-e_1&e_2\\
e_2&0&0
\end{array}
\end{align}
From the fact that 
\begin{align*}
L_{e_1}\circ L_{e_2}-L_{e_2}\circ L_{e_1}&=L_{[e_1,e_2]}=L_{e_2}
\end{align*}
where $L_{e_i}$ is the endomorphism  $L_{e_i}(a)=e_i\star a,$ it follows that \eqref{ConexionHess} defines a left invariant flat affine connection on $\mathsf{Aff}(\mathbb{R})$.\\
Moreover, since $[e_1,e_2]=e_1\star e_2-e_2\star e_1$ we get that the connection $\nabla^H$ is torsion free.\\
Let $\langle\ ,\ \rangle$ denote the left invariant Lorentz metric on $\mathsf{Aff}(\mathbb{R})$ associated to the scalar product on  $\mathfrak{g}$ determined by $\langle e_1 ,e_2 \rangle=1$ and $\langle e_i ,e_i \rangle=0$, $i=1,2$,
an easy calculation shows that 
$$\langle e_i\star e_j, e_k \rangle+\langle e_j,e_i\star e_k \rangle=0$$
for all $i,j,k=1,2.$ Therefore the Hess connection on $\mathsf{Aff}(\mathbb{R})$ is a Lorentzian connection. 
\endproof

\noindent{\textbf{Acknowledgment}}\\
I would like to thank all the members of the Geometry Seminar of the Universidad de Antioquia. They were very helpful at all stages of this work.   

\end{document}